\documentclass[11pt,a4paper]{amsart}

\usepackage[T1]{fontenc}
\usepackage{lmodern}
\usepackage{microtype}
\usepackage[a4paper,margin=1in]{geometry}
\usepackage{mathtools}
\usepackage{amssymb}
\usepackage{booktabs}
\usepackage{float}
\usepackage{tikz}
\usepackage[colorlinks=true,linkcolor=blue!55!black,citecolor=blue!55!black,
            urlcolor=blue!55!black]{hyperref}
\usepackage[nameinlink,capitalize,noabbrev]{cleveref}

\hypersetup{
  pdftitle={Schwarz Symmetrization Can Increase a Nonlocal Threshold Energy},
  pdfauthor={},
  pdfsubject={A counterexample to Schwarz monotonicity for a nonlocal threshold functional},
  pdfkeywords={Schwarz symmetrization, nonlocal threshold energy, counterexample}
}

\newtheorem{theorem}{Theorem}[section]
\newtheorem{lemma}[theorem]{Lemma}
\newtheorem{proposition}[theorem]{Proposition}
\theoremstyle{remark}
\newtheorem{remark}[theorem]{Remark}

\newcommand{\R}{\mathbb{R}}
\newcommand{\1}{\mathbf{1}}
\newcommand{\K}{\mathcal{K}}

\newcommand{\dd}{\,\mathrm{d}}
\newcommand{\doi}[1]{\href{https://doi.org/#1}{doi:#1}}

\title[A nonlocal threshold energy]
{Schwarz Symmetrization Can Increase a Nonlocal Threshold Energy}

\author{Qinyang Li}

\address{School of Mathematical Sciences, Xiamen University, Xiamen 361005, P. R. China}
\email{19020241153813@stu.xmu.edu.cn}

\begin{document}

\begin{abstract}
For $N\geq 1$, $p\geq 1$, and $\delta>0$, consider the nonlocal threshold
functional
\[
  I_{\delta,p}(u)=\iint_{\{|u(x)-u(y)|>\delta\}}
  \frac{\delta^p}{|x-y|^{N+p}}\,\dd x\,\dd y.
\]
Nguyen and Squassina \cite{NguyenSquassina} asked whether $I_{\delta,p}$ decreases under Schwarz
rearrangement.  We give a negative answer in every dimension.  For each
$N\geq1$ and $\delta>0$, one and the same explicit counterexample works for
every $p\geq1$: it is nonnegative, bounded, compactly supported, and takes only
five values.  Its four nontrivial superlevel sets are nested balls whose
centers alternate between two points.  Once the active threshold interactions
are isolated, the energy difference reduces to comparing the interaction of a
unit ball with a concentric annulus and with an eccentric shell.  The sign is
strict because the potential generated by the unit ball decreases with the
radius.  We also compute the energy gap in closed form in dimension one.  This
settles Open Problem~2.2 of Nguyen and Squassina.


\bigskip

\end{abstract}

\maketitle

\section{Introduction}

The classical P\'olya--Szeg\H{o} principle says that symmetric decreasing
rearrangement does not increase the Dirichlet integral.  The same order is
preserved by the Gagliardo seminorm, as follows from the rearrangement
inequalities for convex two-point interactions; see, for example,
\cite{AlmgrenLieb,Baernstein,LiebLoss}.  This fits naturally with the
Bourgain--Brezis--Mironescu formula, which recovers the local Dirichlet energy
from fractional difference quotients
\cite{BBM2001,BBM2002,BrezisNguyen}.  It is therefore tempting to expect a
similar monotonicity for other nonlocal approximations of the same energy.

The approximation relevant here is the Bourgain--Nguyen threshold functional
\begin{equation}\label{eq:functional}
  I_{\delta,p}(u)
  :=\iint_{\{|u(x)-u(y)|>\delta\}}
  \frac{\delta^p}{|x-y|^{N+p}}\,\dd x\,\dd y,
  \qquad \delta>0.
\end{equation}
For every $1<p<\infty$ and $u\in W^{1,p}(\R^N)$, one has
\[
  \lim_{\delta\downarrow0} I_{\delta,p}(u)
  =\frac{K_{N,p}}{p}\int_{\R^N}|\nabla u|^p\,\dd x,
  \qquad
  K_{N,p}:=\int_{\mathbb S^{N-1}}|e\cdot\omega|^p\,\dd\omega,
\]
where $e$ is any unit vector; see
\cite[Theorem~2(c)]{NguyenSobolev}.  Related characterization results appear
in \cite{BourgainNguyen}.  Yet the interaction law in
\eqref{eq:functional} is a sharp threshold rather than a convex function of
$|u(x)-u(y)|$.  Nguyen and Squassina showed that this distinction is already
felt by a single polarization: the energy can increase.  They then asked
whether the full Schwarz rearrangement might nevertheless satisfy
\begin{equation}\label{eq:question}
  I_{\delta,p}(u^*)\leq I_{\delta,p}(u)
\end{equation}
for every nonnegative measurable function $u$; see
\cite[Open Problem~2.2]{NguyenSquassina}.

Why should the usual mechanism fail?  The obstruction is already visible at
the scalar level.  For a cost that is expected to decrease when values are
ordered, $-\Phi$ should be supermodular, or equivalently $\Phi$ should satisfy
the submodular lattice inequality
$\Phi(a,c)+\Phi(b,d)\leq\Phi(a,d)+\Phi(b,c)$ for
$a\leq b$ and $c\leq d$; compare \cite{BurchardHajaiej}.  The threshold cost
$\Phi(a,b)=\1_{\{|a-b|>1\}}$ violates this inequality.  Indeed, for
$a=0$, $b=1/3$, $c=7/6$, and $d=5/3$, one has
$\Phi(a,c)+\Phi(b,d)=2>1=\Phi(a,d)+\Phi(b,c)$.  These four numbers are not
merely an abstract witness: they will reappear as layer values in the example
below.  The remaining difficulty is geometric.  Both entries of the cost must
come from the same function, and all of its superlevel sets must rearrange at
once.  In a different direction, a one-dimensional monotone rearrangement,
applied after vertical $\delta$-segmentation, is central to the
$\Gamma$-convergence analysis of threshold-type interactions in
\cite{AntonucciEtAl}; that argument concerns optimal limiting constants rather
than Schwarz monotonicity of \eqref{eq:functional}.

Our main result resolves \eqref{eq:question} negatively throughout the
exponent range of the polarization theorem in \cite{NguyenSquassina}.  In
fact, the example can be chosen independently of $p$.

\begin{theorem}\label{thm:main}
Let $N\geq1$ and $\delta>0$.  There exists a nonnegative compactly supported
simple function $u_{N,\delta}\in L^\infty(\R^N)$, taking five values, such that,
simultaneously for every $p\geq1$,
\[
  0<I_{\delta,p}(u_{N,\delta})
  <I_{\delta,p}(u_{N,\delta}^*)<\infty.
\]
\end{theorem}

\begin{figure}[H]
\centering
\begin{tikzpicture}[scale=0.82, every node/.style={font=\small},
                    line cap=round]
  \begin{scope}[xshift=-4.8cm]
    \draw[thick] (0,0) circle (2.40);
    \draw[thick] (0,0) circle (1.20);
    \draw[thick,densely dashed] (0.60,0) circle (1.80);
    \draw[thick,densely dashed] (0.60,0) circle (0.60);
    \fill (0,0) circle (1.3pt);
    \fill (0.60,0) circle (1.3pt);
    \draw[->,thin,shorten <=1.5pt,shorten >=1.5pt]
      (0,0) -- node[midway,yshift=-7pt] {$e$} (0.60,0);
    \node at (0,-2.78) {(a) Alternating centers};
    \node[fill=white,inner sep=1pt] at (-1.62,1.30) {$E_1$};
    \node[fill=white,inner sep=1pt] at ( 1.12,1.02) {$E_2$};
    \node[fill=white,inner sep=1pt] at (-0.62,0.68) {$E_3$};
    \node[fill=white,inner sep=1pt] at ( 0.60,0.30) {$E_4$};
  \end{scope}
  \begin{scope}[xshift=3.0cm]
    \draw[thick] (0,0) circle (2.40);
    \draw[thick,densely dashed] (0,0) circle (1.80);
    \draw[thick] (0,0) circle (1.20);
    \draw[thick,densely dashed] (0,0) circle (0.60);
    \fill (0,0) circle (1.3pt);
    \node at (0,-2.78) {(b) Schwarz rearrangement};
    \node[fill=white,inner sep=1pt] at (-1.62,1.30) {$E_1^*$};
    \node[fill=white,inner sep=1pt] at ( 1.12,1.02) {$E_2^*$};
    \node[fill=white,inner sep=1pt] at (-0.62,0.68) {$E_3^*$};
    \node[fill=white,inner sep=1pt] at ( 0,0.30) {$E_4^*$};
  \end{scope}
\end{tikzpicture}
\caption{The alternating-center construction and its Schwarz rearrangement.
In panel~(a), $E_1=B_4(0)$, $E_2=B_3(e)$, $E_3=B_2(0)$, and
$E_4=B_1(e)$, where $|e|=1$.  Thus the solid circles are centered at $0$,
while the dashed circles are centered at $e$.  Panel~(b) shows the
concentric balls $E_j^*=B_{5-j}(0)$.}
\label{fig:geometry}
\end{figure}
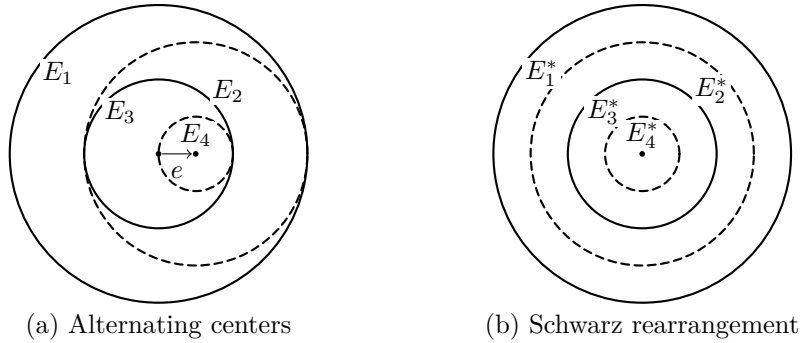

Figure~\ref{fig:geometry} contains the geometric heart of the construction.
We begin with four internally tangent balls of radii $4,3,2,1$, whose centers
alternate between two points at unit distance.  The plateau heights are then
chosen so that every interaction near the diagonal stays below the threshold.
Only three well-separated pairs of layers remain.  Two of them combine into a
term that is already radially arranged and hence does not change.  The entire
energy difference is carried by the third: it compares the innermost ball with
a concentric annulus and with an eccentric shell.  The concentric annulus has
the strictly larger interaction.

We isolate that geometric comparison first.  Once it is available, the
counterexample itself will amount to arranging the five plateau values so that
no other term survives.

\section{An eccentric-shell comparison}

Only one geometric observation is needed.  Outside the unit ball, the
potential generated by that ball decreases strictly with distance from the
origin.  Thus, between two shells of equal volume, the one that keeps more
mass closer to the origin has the larger interaction with $B_1$.  The next
lemma records exactly the configuration that will arise in the proof.

Write $B_r(z)=\{x\in\R^N:|x-z|<r\}$ and abbreviate
$B_r=B_r(0)$.  For measurable sets $A,B\subset\R^N$, set
\begin{equation}\label{eq:setinteraction}
  \K_p(A,B):=\int_A\int_B\frac{1}{|x-y|^{N+p}}\,\dd y\,\dd x
\end{equation}
whenever the integral is finite.

\begin{lemma}\label{lem:shell}
Let $N\geq1$ and $p>0$, and let $e\in\R^N$ be a unit vector.  Then
\begin{equation}\label{eq:shellcomparison}
  \K_p(B_4\setminus B_3,B_1)
  >\K_p(B_4(-e)\setminus B_3,B_1).
\end{equation}
Both interactions in \eqref{eq:shellcomparison} are finite.
\end{lemma}

\begin{proof}
Since $B_3\subset B_4(-e)$, the sets
\[
  F_0:=B_4\setminus B_3,
  \qquad
  F_e:=B_4(-e)\setminus B_3
\]
are genuine shells of the same finite measure.  Both lie at distance at least
$2$ from $B_1$, so their interactions with $B_1$ are finite.

To compare them, introduce the potential generated by $B_1$,
\begin{equation}\label{eq:potential}
  V(x):=\int_{B_1}\frac{1}{|x-y|^{N+p}}\,\dd y.
\end{equation}
It is defined for $|x|>1$, which contains both shells.
Rotational invariance gives $V(x)=v(|x|)$.  For $r>1$ we may differentiate
under the integral sign (locally uniformly on compact subintervals of
$(1,\infty)$) and obtain
\begin{equation}\label{eq:potentialderivative}
  v'(r)=-(N+p)\int_{B_1}
  \frac{r-y_1}{|re_1-y|^{N+p+2}}\,\dd y<0,
\end{equation}
because $r-y_1>0$ for every $y\in B_1$.  Thus $V$ is strictly decreasing
as a function of $|x|$.

The common part $F_0\cap F_e$ contributes equally to the two interactions
and can be discarded.  Set
\[
  C:=F_0\setminus F_e,
  \qquad
  D:=F_e\setminus F_0.
\]
These are the two caps created by translating the outer ball.  The equality
$|F_0|=|F_e|$ gives $|C|=|D|$, and this common value is positive because the
two radius-$4$ balls are distinct.  Since $B_3$ lies in both of them, the cap
removed from the concentric shell lies inside $B_4\setminus B_3$, whereas the
cap added to form the eccentric shell lies outside $B_4$.  Thus, up to null
sets,
\[
  3<|x|<4\quad\text{for a.e. }x\in C,
  \qquad
  |x|>4\quad\text{for a.e. }x\in D.
\]
The removed cap therefore sits where the potential is larger than $v(4)$,
while the added cap sits where it is smaller.  Consequently,
\begin{align*}
  \K_p(F_0,B_1)-\K_p(F_e,B_1)
  &=\int_CV(x)\,\dd x-\int_DV(x)\,\dd x\\
  &>v(4)(|C|-|D|)=0,
\end{align*}
which proves \eqref{eq:shellcomparison}.
\end{proof}

\section{The counterexample}

\begin{proof}[Proof of \cref{thm:main}]
Fix a unit vector $e\in\R^N$ and introduce the four balls
\begin{equation}\label{eq:nestedballs}
  E_1=B_4,
  \qquad E_2=B_3(e),
  \qquad E_3=B_2,
  \qquad E_4=B_1(e).
\end{equation}
The choice of the radii and centers gives
\begin{equation}\label{eq:nesting}
  E_4\subset E_3\subset E_2\subset E_1.
\end{equation}
Indeed, the centers move by one unit each time the radius drops by one.  The
triangle inequality gives every inclusion in \eqref{eq:nesting}, and each pair
of successive balls is internally tangent at one boundary point.

With the geometric comparison already in hand, we now choose the amplitudes so
that it will be the only interaction changed by rearrangement.  Define
$u_{N,\delta}$ (temporarily written as $u$) by
\begin{equation}\label{eq:counterexample}
  u:=\delta\left(
  \frac13\1_{E_1}+\frac12\1_{E_2}
  +\frac13\1_{E_3}+\frac12\1_{E_4}
  \right).
\end{equation}
The alternating increments $1/3,1/2,1/3,1/2$ are designed to separate the
amplitude bookkeeping from the geometry.  To see this, let
\[
  A_0:=\R^N\setminus E_1,
  \quad A_1:=E_1\setminus E_2,
  \quad A_2:=E_2\setminus E_3,
  \quad A_3:=E_3\setminus E_4,
  \quad A_4:=E_4.
\]
The values of $u$ on the five layers are
\begin{equation}\label{eq:layervalues}
\begin{array}{c|ccccc}
  \text{layer} & A_0&A_1&A_2&A_3&A_4\\ \midrule
  u/\delta &0&\frac13&\frac56&\frac76&\frac53.
\end{array}
\end{equation}
Reading across this table, the unordered pairs of layers whose values differ
by more than $\delta$ are exactly
\begin{equation}\label{eq:activepairs}
  (A_0,A_3),\qquad (A_0,A_4),\qquad (A_1,A_4).
\end{equation}
This is the decisive bookkeeping point.  Every pair of layers whose closures
meet has value gap at most $5\delta/6$, and in fact every inactive difference
is at most $5\delta/6$.  The threshold therefore removes all near-diagonal
interactions and leaves only the three spatially separated pairs in
\eqref{eq:activepairs}.

Let $p\geq1$ be arbitrary; no choice made above depends on it.  The first two
active pairs share $A_0$, and $A_3\cup A_4=E_3$.  Using also the symmetry of
the ordered double integral, we obtain
\begin{equation}\label{eq:energyu}
  \frac{I_{\delta,p}(u)}{2\delta^p}
  =\K_p(E_1^c,E_3)+\K_p(E_1\setminus E_2,E_4).
\end{equation}
Both pairs of sets on the right-hand side are separated by distance at least
$2$.  The second interaction is bounded and is therefore immediately finite.
The first reaches infinity, so we record the elementary tail estimate
\begin{align}\label{eq:tailbound}
 \K_p(B_4^c,B_2)
 &\leq |B_2|\left[
 2^{-N-p}|B_8\setminus B_4|
 +2^{N+p}\int_{|x|\geq8}|x|^{-N-p}\,\dd x\right]<\infty.
\end{align}
Here we used $|x-y|\geq |x|/2$ when $|x|\geq8$ and $y\in B_2$; the final
integral converges because $p>0$.  Thus both terms in \eqref{eq:energyu} are
finite, and both are plainly positive.

Notice that the first term already involves the concentric pair
$E_1=B_4$ and $E_3=B_2$; it will remain unchanged under rearrangement.  To
identify what happens to the second term, we read the rearranged function from
its superlevel sets.  Apart from the four plateau heights,
\[
  \{u>t\}=\begin{cases}
  E_1,&0<t<\delta/3,\\
  E_2,&\delta/3<t<5\delta/6,\\
  E_3,&5\delta/6<t<7\delta/6,\\
  E_4,&7\delta/6<t<5\delta/3.
  \end{cases}
\]
Since $|E_j|=|B_{5-j}|$ for $j=1,\ldots,4$, the equimeasurability
characterization of Schwarz rearrangement gives
\begin{equation}\label{eq:rearrangedu}
  u^*=\delta\left(
  \frac13\1_{B_4}+\frac12\1_{B_3}
  +\frac13\1_{B_2}+\frac12\1_{B_1}
  \right).
\end{equation}
Repeating the layer computation gives
\begin{equation}\label{eq:energyustar}
  \frac{I_{\delta,p}(u^*)}{2\delta^p}
  =\K_p(B_4^c,B_2)+\K_p(B_4\setminus B_3,B_1).
\end{equation}
The promised cancellation is now visible: the first terms in
\eqref{eq:energyu} and \eqref{eq:energyustar} agree because $E_1=B_4$ and
$E_3=B_2$.  The whole energy difference is therefore carried by the two shell
interactions.  Translating both variables by $-e$ in the remaining term of
\eqref{eq:energyu}, we find
\[
  \K_p(E_1\setminus E_2,E_4)
  =\K_p(B_4(-e)\setminus B_3,B_1).
\]
Consequently, \cref{lem:shell} gives
\begin{align*}
  \frac{I_{\delta,p}(u^*)-I_{\delta,p}(u)}{2\delta^p}
  &=\K_p(B_4\setminus B_3,B_1)
    -\K_p(B_4(-e)\setminus B_3,B_1)\\
  &>0.
\end{align*}
Since $p$ was arbitrary, the same $u_{N,\delta}$ works simultaneously for
every $p\geq1$.  This proves the theorem.
\end{proof}

\section{An exact one-dimensional check}

No separate argument is needed in dimension one, but the interval model makes
the cancellation completely visible.  It also provides a useful check on the
factor $2$ coming from the ordered double integral.  Taking $e=-1$ in
\eqref{eq:nestedballs}, translating by $4$, and setting $\delta=1$, we obtain
\begin{equation}\label{eq:onedu}
u(x)=
\begin{cases}
  \frac56,&0<x<2,\\
  \frac53,&2<x<4,\\
  \frac76,&4<x<6,\\
  \frac13,&6<x<8,\\
  0,&\text{otherwise}.
\end{cases}
\end{equation}
Its symmetric decreasing rearrangement is
\begin{equation}\label{eq:onedustar}
u^*(x)=
\begin{cases}
  \frac53,&|x|<1,\\
  \frac76,&1<|x|<2,\\
  \frac56,&2<|x|<3,\\
  \frac13,&3<|x|<4,\\
  0,&|x|>4.
\end{cases}
\end{equation}

\begin{proposition}\label{prop:oned}
For the functions in \eqref{eq:onedu}--\eqref{eq:onedustar} and $p>1$,
\begin{equation}\label{eq:gapformula}
  I_{1,p}(u^*)-I_{1,p}(u)
  =\frac{2}{p(p-1)}
  \left(2^{1-p}-2\,3^{1-p}+2\,5^{1-p}-6^{1-p}\right)>0.
\end{equation}
At $p=1$ the gap is
\begin{equation}\label{eq:p1gap}
  I_{1,1}(u^*)-I_{1,1}(u)=2\log\frac{27}{25}>0.
\end{equation}
In particular, when $p=2$,
\[
  I_{1,2}(u)=\frac56,
  \qquad I_{1,2}(u^*)=\frac9{10},
  \qquad I_{1,2}(u^*)-I_{1,2}(u)=\frac1{15}.
\]
\end{proposition}

\begin{proof}
For $a<b<c<d$ and $p>1$, direct integration gives
\begin{equation}\label{eq:intervalinteraction}
\int_a^b\int_c^d\frac{1}{(y-x)^{1+p}}\,\dd y\,\dd x
=\frac{(c-b)^{1-p}+(d-a)^{1-p}-(c-a)^{1-p}-(d-b)^{1-p}}
{p(p-1)}.
\end{equation}
The corresponding half-line formulas follow by letting $a\to-\infty$ or
$d\to\infty$.  In the coordinates of \eqref{eq:onedu}, the three active
layer pairs give
\[
  \frac{I_{1,p}(u)}{2}
  =\K_p\big((-\infty,0)\cup(8,\infty),(2,6)\big)
   +\K_p\big((2,4),(6,8)\big),
\]
whereas the rearranged layers give
\[
  \frac{I_{1,p}(u^*)}{2}
  =\K_p\big((-\infty,-4)\cup(4,\infty),(-2,2)\big)
   +\K_p\big((-4,-3)\cup(3,4),(-1,1)\big).
\]
Applying \eqref{eq:intervalinteraction}, together with its half-line limits,
and remembering that the ordered double integral counts both orientations,
gives
\begin{align}
 I_{1,p}(u)
 &=\frac{2}{p(p-1)}
   \left(3\,2^{1-p}-2\,4^{1-p}-6^{1-p}\right),
   \label{eq:onedenergyu}\\
 I_{1,p}(u^*)
 &=\frac{4}{p(p-1)}
   \left(2\,2^{1-p}-3^{1-p}-4^{1-p}
         +5^{1-p}-6^{1-p}\right).
   \label{eq:onedenergyustar}
\end{align}
Subtracting these identities yields \eqref{eq:gapformula}; in particular, the
terms containing $4^{1-p}$ cancel.

We still have to see why the resulting expression is positive.  Put
$q=p-1>0$.  The expression in
parentheses in \eqref{eq:gapformula} equals
\begin{align}\label{eq:convexidentity}
q\int_0^1\big[&(2+t)^{-q-1}+(5+t)^{-q-1}\nonumber\\
              &-(3+t)^{-q-1}-(4+t)^{-q-1}\big]\,\dd t.
\end{align}
For every fixed $t\in[0,1]$, the two pairs $2+t,5+t$ and $3+t,4+t$
have the same sum, while the first pair is strictly more spread out.
The function $s\mapsto s^{-q-1}$ is strictly convex.  Hence the integrand in
\eqref{eq:convexidentity} is strictly positive, which proves the sign in
\eqref{eq:gapformula}.

It remains to pass to $p=1$.  All active pairs are separated by distance at
least $2$.  Hence, for $1\leq p\leq2$, their kernels are dominated by the
integrable $p=1$ kernel.  Dominated convergence gives
$I_{1,p}(u)\to I_{1,1}(u)$ and
$I_{1,p}(u^*)\to I_{1,1}(u^*)$.  If
\[
  F(q):=2^{-q}-2\,3^{-q}+2\,5^{-q}-6^{-q},
\]
then $F(0)=0$ and
\[
  F'(0)=-\log2+2\log3-2\log5+\log6
       =\log\frac{27}{25}.
\]
Letting $p\downarrow1$ in \eqref{eq:gapformula} now gives
\eqref{eq:p1gap}.  Finally, setting $p=2$ in
\eqref{eq:onedenergyu}--\eqref{eq:onedenergyustar} gives
$I_{1,2}(u)=5/6$ and $I_{1,2}(u^*)=9/10$.
\end{proof}

\section{Concluding remarks}

\begin{remark}[Why the energy stays finite]
Although the counterexample is discontinuous, every pair of layers whose
closures meet has amplitude difference at most $5\delta/6$.  Thus no pair
selected by the threshold can approach the diagonal; indeed, every active
pair is separated by distance at least $2$.  The singularity of
$|x-y|^{-N-p}$ is therefore never sampled.  The sole unbounded interaction
has an integrable
$|x|^{-N-p}$ tail, as quantified in \eqref{eq:tailbound}; hence both energies
are finite.
\end{remark}

\begin{remark}[The threshold convention]
None of the differences between plateau values is exactly $\delta$.
Accordingly, the same counterexample works if the condition $>\delta$ in
\eqref{eq:functional} is replaced by $\geq\delta$.
\end{remark}

\begin{remark}[Compatibility with the limiting formula]
The construction scales with the threshold: for fixed $N$,
$u_{N,\delta}=\delta u_{N,1}$.  Thus the function itself varies with
$\delta$.  There is no conflict with the pointwise limit formula for a fixed
Sobolev function as $\delta\downarrow0$, nor with the
classical P\'olya--Szeg\H{o} inequality for the limiting Dirichlet energy.
\end{remark}

\begin{remark}[Range of exponents]
The proof of \cref{lem:shell} and the finiteness argument in fact work for
every $p>0$.  We stated \cref{thm:main} for $p\geq1$ to match the exponent
range of the polarization theorem in \cite{NguyenSquassina}.  In particular,
the counterexample covers the motivating regime $1<p<N$ whenever that range
is nonempty.
\end{remark}

\subsection*{Data availability}
No data were used for the research described in this article.

\subsection*{Competing interests}
The authors declare that they have no competing interests.

\subsection*{AI assistance statement}
The authors used OpenAI models to assist with language polishing.


\begin{thebibliography}{99}

\bibitem{AlmgrenLieb}
F.~J. Almgren, Jr. and E.~H. Lieb,
\emph{Symmetric decreasing rearrangement is sometimes continuous},
J. Amer. Math. Soc. \textbf{2} (1989), no.~4, 683--773,
\doi{10.1090/S0894-0347-1989-1002633-4}.

\bibitem{AntonucciEtAl}
C.~Antonucci, M.~Gobbino, M.~Migliorini, and N.~Picenni,
\emph{Optimal constants for a nonlocal approximation of Sobolev norms and
total variation}, Anal. PDE \textbf{13} (2020), no.~2, 595--625,
\doi{10.2140/apde.2020.13.595}.

\bibitem{Baernstein}
A.~Baernstein II,
\emph{A unified approach to symmetrization},
in A.~Alvino, E.~Fabes, and G.~Talenti (eds.),
\emph{Partial Differential Equations of Elliptic Type} (Cortona, 1992),
Sympos. Math. XXXV, Cambridge Univ. Press, Cambridge, 1994, pp.~47--91.

\bibitem{BBM2001}
J.~Bourgain, H.~Brezis, and P.~Mironescu,
\emph{Another look at Sobolev spaces},
in J.~L. Menaldi, E.~Rofman, and A.~Sulem (eds.),
\emph{Optimal Control and Partial Differential Equations: A Volume in Honour
of A.~Bensoussan's 60th Birthday}, IOS Press, Amsterdam, 2001, pp.~439--455.

\bibitem{BBM2002}
J.~Bourgain, H.~Brezis, and P.~Mironescu,
\emph{Limiting embedding theorems for $W^{s,p}$ when $s\uparrow1$ and applications},
J. Anal. Math. \textbf{87} (2002), 77--101,
\doi{10.1007/BF02868470}.

\bibitem{BourgainNguyen}
J.~Bourgain and H.-M. Nguyen,
\emph{A new characterization of Sobolev spaces},
C. R. Math. Acad. Sci. Paris \textbf{343} (2006), no.~2, 75--80,
\doi{10.1016/j.crma.2006.05.021}.

\bibitem{BrezisNguyen}
H.~Brezis and H.-M. Nguyen,
\emph{The BBM formula revisited},
Atti Accad. Naz. Lincei Rend. Lincei Mat. Appl. \textbf{27} (2016), no.~4,
515--533, \doi{10.4171/RLM/746}.

\bibitem{BurchardHajaiej}
A.~Burchard and H.~Hajaiej,
\emph{Rearrangement inequalities for functionals with monotone integrands},
J. Funct. Anal. \textbf{233} (2006), no.~2, 561--582,
\doi{10.1016/j.jfa.2005.08.010}.

\bibitem{LiebLoss}
E.~H. Lieb and M.~Loss,
\emph{Analysis}, 2nd ed., Graduate Studies in Mathematics, vol.~14,
American Mathematical Society, Providence, RI, 2001,
\doi{10.1090/gsm/014}.

\bibitem{NguyenSobolev}
H.-M. Nguyen,
\emph{Some new characterizations of Sobolev spaces},
J. Funct. Anal. \textbf{237} (2006), no.~2, 689--720,
\doi{10.1016/j.jfa.2006.04.001}.

\bibitem{NguyenSquassina}
H.-M. Nguyen and M.~Squassina,
\emph{Some remarks on rearrangement for nonlocal functionals},
Nonlinear Anal. \textbf{162} (2017), 1--12,
\doi{10.1016/j.na.2017.06.007}.

\end{thebibliography}
\end{document}